\documentclass{article}

\usepackage{graphicx}
\usepackage{xcolor}
\usepackage{indentfirst}
\usepackage{amsmath,amsfonts,amsthm,amssymb,amsrefs}
\usepackage{mathrsfs,stmaryrd}
\usepackage[all]{xy}
\usepackage{hyperref}

\def\co{\colon\thinspace}
\DeclareMathAlphabet{\mathsfsl}{OT1}{cmss}{m}{sl}
\newcommand{\tensor}[1]{\mathsfsl{#1}}
\newcommand{\spinc}{\mathrm{Spin}^c}

\newcommand{\im}{\mathrm{im}}

\newcommand{\bhfminus}{\mathbf{HF}^-}

\newcommand{\bhfle}{\mathbf{HF}^{\le0}}
\newcommand{\bFle}{\mathbf{F}^{\le0}}

\newtheorem{thm}{Theorem}[section]
\newtheorem{lem}[thm]{Lemma}

\newtheorem{prop}[thm]{Proposition}

\theoremstyle{definition}

\newtheorem{rem}[thm]{Remark}

\begin{document}

\title{The next-to-top term in knot Floer homology}

\author{
{\Large Yi NI}\\{\normalsize Department of Mathematics, Caltech, MC 253-37}\\
{\normalsize 1200 E California Blvd, Pasadena, CA
91125}\\{\small\it Emai\/l\/:\quad\rm yini@caltech.edu}}

\date{}
\maketitle

\begin{abstract}
Let $K$ be a null-homologous knot in a generalized L-space $Z$ with $b_1(Z)\le1$. Let $F$ be a Seifert surface of $K$ with genus $g$. We show that if $\widehat{HFK}(Z,K,[F],g)$ is supported in a single $\mathbb Z/2\mathbb Z$--grading, then \[\mathrm{rank}\widehat{HFK}(Z,K,[F],g-1)\ge\mathrm{rank}\widehat{HFK}(Z,K,[F],g).\] 
\end{abstract}

\section{Introduction}

Knot Floer homology is an invariant for null-homologous knots in $3$--manifolds introduced by Ozsv\'ath--Szab\'o \cite{OSzKnot} and Rasmussen \cite{RasThesis}. Suppose that $F$ is a Thurston norm minimizing Seifert surface for a null-homologous knot $K\subset Z$, then $\widehat{HFK}(Z,K,[F],g(F))$, which is known as ``the topmost term'' in knot Floer homology, captures a lot of information about the knot complement. For example, $\widehat{HFK}(Z,K,[F],g(F))$ always has positive rank \cite{OSzGenus}. Moreover, $\widehat{HFK}(Z,K,[F],g(F))$ has rank $1$ if and only if $F$ is a fiber of a fibration of $Z\setminus K$ over $S^1$ \cite{Gh,NiFibred}. 

It is natural to ask if one can say similar things for other terms in $\widehat{HFK}(Z,K)$. Baldwin and Vela-Vick \cite[Question~1.11]{BVV} asked whether $\widehat{HFK}(S^3,K,g(K)-1)$ is always nontrivial. More specifically, Sivek \cite[Question~1.12]{BVV} asked whether we always have 
\begin{equation}\label{eq:SivekConj}
\mathrm{rank}\widehat{HFK}(S^3,K,g(K)-1)\ge\mathrm{rank}\widehat{HFK}(S^3,K,g(K)).
\end{equation}
This inequality has been known for knots with thin knot Floer homology \cite{OSzAlternating}, L-space knots \cite{HW}, fibered knots in any closed oriented $3$--manifolds \cite{BVV}. In this paper, we will prove (\ref{eq:SivekConj}) when $\widehat{HFK}(Z,K,[F],g)$ is supported in a single $\mathbb Z/2\mathbb Z$--grading.

Recall that a closed, oriented $3$--manifold $Z$ is a {\it generalized L-space} if \[HF_{\mathrm{red}}(Z)=0.\] In \cite{OSzAnn2}, an absolute $\mathbb Z/2\mathbb Z$--grading was defined on Heegaard Floer homology. When the underlying Spin$^c$ structure is torsion, one can define an absolute $\mathbb Q$--grading.

\begin{thm}\label{thm:SecondTerm}
Let $Z$ be a generalized L-space with $b_1(Z)\le1$, and let $K\subset Z$ be a null-homologous knot with a Thurston norm minimizing Seifert surface $F$ of genus $g>0$. Suppose that $\widehat{HFK}(Z,K,[F],g)$ is supported in a single $\mathbb Z/2\mathbb Z$--grading. Then for any $d\in\mathbb Q$, we have
\[
\mathrm{rank}\widehat{HFK}_{d-1}(Z,K,[F],g-1)\ge\mathrm{rank}\widehat{HFK}_d(Z,K,[F],g).
\]
\end{thm}

To prove Theorem~\ref{thm:SecondTerm}, we need the following result about $HF^+$.

\begin{thm}\label{thm:ClosedSecond}
Let $Y$ be a closed oriented $3$--manifold. Suppose that $G\subset Y$ is a closed oriented surface of genus $g>2$. If 
there exist two elements $\gamma_1,\gamma_2\in H_1(G)$ with $\gamma_1\cdot\gamma_2\ne0$, such that their images in $H_1(Y)$ are linearly dependent,
then the map $U$ is trivial on $HF^+(Y,[G],g-2;\mathbb Q)$.
\end{thm}

\begin{rem}
When $b_1(Y)\le 2$, a simple intersection number argument shows that the image of $H_1(G;\mathbb Q)\to H_1(Y;\mathbb Q)$ is at most $1$--dimensional for any $G\subset Y$ with $[G]\ne0\in H_2(Y)$. So Theorem~\ref{thm:ClosedSecond} can be applied to this case. Ozsv\'ath and Szab\'o have computed $HF^+(S^3_0(K))$ in the cases when $K$ is an L-space knot \cite[Proposition~8.1]{OSzAbGr} and when $K$ is an alternating knot \cite[Theorem~1.4]{OSzAlternating}. One can directly check Theorem~\ref{thm:ClosedSecond} in these two cases.
\end{rem}

\begin{rem}
If $G\subset Y$ is a closed oriented surface of genus $g>1$, the map $U$ on $HF^+(Y,[G],g-1)$ is trivial. The author first learned this result from Peter Ozsv\'ath, and learned a sketch of a proof of it from Yank{\i} Lekili using a similar argument as in \cite[Theorem~3.1]{OSzSympl}. A proof of a more general result using the same idea as Lekili's was given by Wu \cite{Wu}. The proof of Theorem~\ref{thm:ClosedSecond} uses the same argument. Our proof justifies the use of the K\"unneth formula for $HF^+$ in \cite{Wu}. 
\end{rem}

This paper is organized as follows. In Section~\ref{sect:Prelim}, we will collect some results about Heegaard Floer homology we will use. In Section~\ref{sect:HF+}, we prove Theorem~\ref{thm:ClosedSecond}. In Section~\ref{sect:Main}, we prove Theorem~\ref{thm:SecondTerm}. 

We will use the following notations in this paper. If $N$ is a submanifold of another manifold $M$, let $\nu(N)$ be a closed tubular neighborhood of $N$ in $M$, and let $\nu^{\circ}(N)$ be the interior of $\nu(N)$. If $K$ is a null-homologous knot in a $3$--manifold $Z$, let $Z_{p/q}(K)$ be the manifold obtained by $\frac pq$--surgery on $K$.

\vspace{5pt}\noindent{\bf Acknowledgements.}\quad  The author was
partially supported by NSF grant number DMS-1811900. The author is indebted to Robert Lipshitz for many fruitful discussions and critical comments which shaped this work.


\section{Preliminaries on Heegaard Floer homology}\label{sect:Prelim}

Heegaard Floer homology \cite{OSzAnn1}, in its most fundamental form, assigns a package of invariants
$$\widehat{HF}, HF^+, HF^-, HF^{\infty}$$
to a closed, connected, oriented $3$--manifold $Y$ equipped with a Spin$^c$ structure $\mathfrak s\in\spinc(Y)$.

As in \cite{MO}, let $\bhfminus$ and $\mathbf{HF}^{\infty}$ denote the completions of $HF^-$ and $HF^{\infty}$ with respect to the maximal ideal
$(U)$ in the ring $\mathbb Z[U]$. By \cite[Equations~(5) and~(6)]{MO}, when $c_1(\mathfrak s)$ is non-torsion, $\mathbf{HF}^{\infty}(Y,\mathfrak s)=0$, so 
\begin{equation}\label{eq:+-Isom}
{HF}^+(Y,\mathfrak s)\cong \bhfminus(Y,\mathfrak s).
\end{equation}

Let $CF^{\le0}(Y,\mathfrak s)$ be the subcomplex of $CF^{\infty}(Y,\mathfrak s)$ which consists of $[\mathbf x,i]$, $i\le0$. This chain complex is clearly isomorphic to
$CF^{-}(Y,\mathfrak s)$ via the $U$--action. We have a similar completion $\bhfle$.

We often use $HF^{\circ}$ to denote one of the above invariants.

When $W$ is a cobordism from $Y_1$ to $Y_2$, and $\mathfrak S\in\spinc(W)$, there is an induced homomorphism
\[
F^{\circ}_{W,\mathfrak S}\co HF^{\circ}(Y_1,\mathfrak S|_{Y_1})\to HF^{\circ}(Y_2,\mathfrak S|_{Y_2}).
\]

Given $\gamma\in H_1(Y)/\mathrm{Tors}$, one can define a homomorphism \[A_{\gamma}\co HF^{\circ}(Y)\to HF^{\circ}(Y)\] satisfying $A_{\gamma}^2=0$.
The following theorem is the $\bhfle$ version of \cite[Theorem~3.6]{HeddenNi}. See the paragraph after it.

\begin{thm}\label{thm:Cob}
Suppose $Y_1,Y_2$ are two closed, oriented, connected
$3$--manifolds, and $W$ is a cobordism from $Y_1$ to $Y_2$. Let
$$\bFle_{W}\co \bhfle(Y_1)\longrightarrow
\bhfle(Y_2)$$ be the homomorphism induced by
$W$. Suppose $\zeta_1\subset Y_1$, $\zeta_2\subset Y_2$ are two
closed curves which are homologous in $W$. Then 
$$\bFle_{W}\circ A_{[\zeta_1]}= A_{[\zeta_2]}\circ \bFle_{W}.$$
\end{thm}


\section{The next-to-top term in $HF^+$}\label{sect:HF+}

We will use $\mathbb Q$--coefficients for Heegaard Floer homology in the rest of this paper. 

Let $G$ be a closed oriented surface of genus $g>2$.
Let \[V\co S^3\to G\times S^1\] be the cobordism which consists of $2g$ one-handles and $1$ two-handle with attaching curve being the Borromean knot $B_g$. Let $\mathfrak S_{g-2}\in\spinc(V)$ be the Spin$^c$ structure with $\langle c_1(\mathfrak S_{g-2}),[G]\rangle=2g-4$, and let $\mathfrak s_{g-2}\in\spinc(G\times S^1)$ be the restriction of $\mathfrak S_{g-2}$ to $G\times S^1$.

Let \[\bFle_{V,\mathfrak S_{g-2}}\co \bhfle(S^3)\to \bhfle(G\times S^1,\mathfrak s_{g-2})\] be the map induced by the cobordism $(V,\mathfrak S_{g-2})$, and let 
\begin{equation}\label{eq:y}
y=\bFle_{V,\mathfrak S_{g-2}}(\mathbf 1).
\end{equation}

In \cite[Theorem~9.3]{OSzKnot}, it is shown that
\begin{equation}\label{eq:HF+Prod}
HF^{+}(G\times S^1,\mathfrak s_{g-2})\cong X(g,1)=H^0(G)\otimes\mathbb Q[U]/(U^2)\oplus H^1(G)\otimes \mathbb Q[U]/(U),
\end{equation}
with the homological action given by
\begin{equation}\label{eq:AProd}
A_{\gamma}(\theta\otimes1)=\mathrm{PD}(\gamma)\otimes\mathbf 1, \quad A_{\gamma}(\eta\otimes1)=\langle\eta,\gamma\rangle\otimes U.
\end{equation}
Here $\theta$ is a generator of $H^0(G)$, and $\eta\in H^1(G)$.
We will fix an identification as in (\ref{eq:HF+Prod}). By abuse of notation, we often use $\theta$ to denote $\theta\otimes1\in X(g,1)$.

We will prove the following proposition.

\begin{prop}\label{prop:yValue}
The element $y$ defined in (\ref{eq:y}) has the form $a\theta+bU\theta$ for some $a,b\in\mathbb Q$, $a\ne0$.
\end{prop}

Let $Y$ be a closed, oriented $3$--manifold and suppose that $G$ embeds into $Y$ as a homologically essential surface.
Consider the trivial cobordism
\[
Y\times[0,1]\co Y\to Y.
\]
Let $\mathsf p$ be a point in $G$, and let $W_1$ be a tubular neighborhood of \[(Y\times\{0\})\cup(\mathsf p\times[0,\frac12])\cup (G\times\{\frac12\}).\] Then $W_1$ is a cobordism from $Y$ to $Y\#(G\times S^1)$. Let $W_2=\overline{Y\times[0,1]\setminus W_1}$. 

Let $\mathfrak t\in\spinc(Y)$ be a Spin$^c$ structure satisfying $\langle c_1(\mathfrak t),[G]\rangle=2(g-2)$, and let $\mathfrak T\in\spinc(Y\times[0,1])$ be the corresponding Spin$^c$ structure. If we think of $G\times S^1$ as the boundary of a regular neighborhood of $G\times\{\frac12\}$, then we clearly have $\mathfrak T|_{G\times S^1}=\mathfrak s_{g-2}$.
By \cite[Lemma~2.1]{NiPropG},
\begin{equation}\label{eq:CompId}
F^{\circ}_{W_2,\mathfrak T|_{W_2}}\circ F^{\circ}_{W_1,\mathfrak T|_{W_1}}=\mathrm{id}\co HF^{\circ}(Y,\mathfrak t)\to HF^{\circ}(Y,\mathfrak t).
\end{equation}

\begin{lem}\label{lem:Fx}
Suppose that $x\in \bhfle(Y,\mathfrak t)$, then $\bFle_{W_1,\mathfrak T|_{W_1}}(x)=x\otimes y$. Here $y$ is defined in (\ref{eq:y}), and 
\[
x\otimes y\in  \bhfle(Y,\mathfrak t)\otimes_{\mathbb Q[U]}\bhfle(G\times S^1,\mathfrak s_{g-2})\subset \bhfle(Y\#(G\times S^1),\mathfrak t\#\mathfrak s_{g-2})
\]
by the K\"unneth formula.
\end{lem}
\begin{proof}
By \cite[Proposition~4.4]{OSzAbGr}, there is a commutative diagram (note that we switch the order of the tensor product)
\[
\xymatrixcolsep{6pc}
\xymatrix{ \bhfle(Y,\mathfrak t)\otimes \bhfle(S^3)\ar[r]^-{\bFle_{Y\#S^3,\mathfrak t}}
\ar[d]^-{\mathrm{id}\otimes \bFle_{V,\mathfrak S_{g-2}}}& \bhfle(Y,\mathfrak t)\ar[d]^-{\bFle_{W_1,\mathfrak T|_{W_1}}}\\
 \bhfle(Y,\mathfrak t)\otimes \bhfle(G\times S^1,\mathfrak s_{g-2}) \ar[r]^-{\bFle_{Y\#(G\times S^1),\mathfrak t\#\mathfrak s_{g-2}}} & \bhfle(Y\#(G\times S^1),\mathfrak t\#\mathfrak s_{g-2}).
}
\]
Our conclusion follows from this commutative diagram.
\end{proof}

\begin{proof}[Proof of Proposition~\ref{prop:yValue}]
We choose $Y=G\times S^1$ and $x=U\theta$. By (\ref{eq:CompId}) and  Lemma~\ref{lem:Fx},
\[
U\theta=\bFle_{W_2}\circ \bFle_{W_1}(U\theta)=\bFle_{W_2}(U\theta\otimes y)=\bFle_{W_2}(\theta\otimes Uy).
\]
Since $U\theta\ne0$, $Uy\ne0$. From the structure of $X(g,1)$ in (\ref{eq:HF+Prod}), we see that any homogeneous element $y$ (with respect to the $\mathbb Z/2\mathbb Z$--grading) satisfying $Uy\ne0$ must be of the form $a\theta+bU\theta$, $a\ne0$.
\end{proof}

\begin{lem}\label{lem:Image1}
For any $\gamma_1,\gamma_2\in H_1(G)\subset H_1(G\times S^1)$,
we have \[A_{\gamma_2}\circ A_{\gamma_1}(y)=(\gamma_1\cdot\gamma_2)Uy.\]
\end{lem}
\begin{proof}
By Proposition~\ref{prop:yValue}, $y=a\theta+bU\theta$. By the module structure of $X(g,1)$ in (\ref{eq:HF+Prod}) and (\ref{eq:AProd}), $Uy=aU\theta$, and 
\[A_{\gamma_2}\circ A_{\gamma_1}(y)=\langle\mathrm{PD}(\gamma_1),\gamma_2\rangle aU\theta=(\gamma_1\cdot\gamma_2)aU\theta.\qedhere\]
\end{proof}

\begin{proof}[Proof of Theorem~\ref{thm:ClosedSecond}]
Let $\mathfrak t\in\spinc(Y)$ be as above.
Assume that $U\ne0$ on $HF^+(Y,\mathfrak t)$. By (\ref{eq:+-Isom}), $Ux\ne0$ for some $x\in \bhfle(Y,\mathfrak t)$.
By (\ref{eq:CompId}) and  Lemma~\ref{lem:Fx},
\begin{equation}\label{eq:xIm}
x=\bFle_{W_2}\circ \bFle_{W_1}(x)=\bFle_{W_2}(x\otimes y).
\end{equation}
Let $c_i\subset G$ be a closed curve representing $\gamma_i$, $i=1,2$. Let $\gamma_i'\in H_1(Y\#(G\times S^1))$ be represented by $c_i\times\text{point}\subset G\times S^1$, and let $\gamma_i''\in H_1(Y)$ be represented by $c_i\subset G\subset Y$. 
Then $(c_i\times[\frac12,1])\cap W_2$ defines a homology between $\gamma_i'$ and $\gamma_i''$.
By Lemma~\ref{lem:Image1} and (\ref{eq:xIm}) we have
\begin{eqnarray*}
(\gamma_1\cdot\gamma_2)Ux&=&
\bFle_{W_2}(x\otimes (\gamma_1\cdot\gamma_2)Uy)\\
&=&\bFle_{W_2}(x\otimes A_{\gamma_2}\circ A_{\gamma_1}(y))\\
&=&\bFle_{W_2}(A_{\gamma_2'}\circ A_{\gamma_1'}(x\otimes y)),
\end{eqnarray*}
where the last equality follows from the fact that the actions of $A_{\gamma_1'}$ and $A_{\gamma_2'}$ on the $\bhfle(Y,\mathfrak t)$ factor are trivial.

Since $\gamma_1''$ and $\gamma_2''$ in $H_1(Y)$ are linearly dependent, we get
\[\bFle_{W_2}(A_{\gamma_2'}\circ A_{\gamma_1'}(x\otimes y))=A_{\gamma_2''}\circ A_{\gamma_1''}\bFle_{W_2}(x\otimes y)=0\] by 
Theorem~\ref{thm:Cob} and the fact that $A_{\gamma}^2=0$ for any $\gamma\in H_1(Y)$. This contradicts the assumption that $\gamma_1\cdot\gamma_2\ne0$ and $Ux\ne0$.
\end{proof}


\section{Proof of the main theorem}\label{sect:Main}

Let $K$ be a null-homologous knot in a generalized L-space $Z$. Let $F$ be a Thurston norm minimizing Seifert surface of $K$ with genus $g>2$. 
Let $C=CFK^{\infty}(Z,K,[F])$.
Let 
\[A_k^+=C\{i\ge0 \text{ or } j\ge k\}, B^+=C\{i\ge0 \}\]
and define maps 
\[
v_k^+, h_k^+\co A_k^+\to B^+
\]
as in \cite{OSzIntSurg}. By \cite[Theorem~7.1]{OSzFour}, the difference between the grading shifts of $v_{k}^+$ and $h_{k}^+$ is
\begin{equation}\label{eq:vhDiff}
-\frac{(2k-n)^2-(2k+n)^2}{4n}=2k.
\end{equation}

\begin{prop}\label{prop:ZeroSurg}
Let  $\widehat F$ be the closed surface in $Z_0(K)$ obtained by capping off $\partial F$ with a disk.
If there exists an element $a\in H_*(C\{i<0, j\ge g-2\})$ such that $Ua\ne0$, then there also exists an element $a'\in HF^+(Z_0(K),[\widehat{F}],g-2)$ such that $Ua'\ne0$.
\end{prop}
\begin{proof}
Consider the short exact sequence of chain complexes
\begin{equation}\label{eq:vShort}
\xymatrix{
0\ar[r] &C\{i<0, j\ge g-2\}\ar[r] &A_{g-2}^+\ar[r]^{v_{g-2}^+}&B^+ \ar[r] &0,
}
\end{equation}
which induces an exact triangle. Since $Z$ is a generalized L-space, 
\[v=(v_{g-2}^+)_*\co H_*(A_{g-2}^+)\to H_*(B^+)\]
is surjective. So 
\[H_*(C\{i<0, j\ge g-2\})\cong\ker v
\]
as a $\mathbb Q[U]$--module.

By \cite[Subsection~4.8]{OSzIntSurg}, $CF^+(Z_0(K),[\widehat{F}],g-2)$ is quasi-isomorphic to the mapping cone of 
\[
v_{g-2}^++ h_{g-2}^+\co A_k^+\to B^+.
\]
By (\ref{eq:vhDiff}), $v_{g-2}^+$ and $h_{g-2}^+$ have different grading shifts. Since $Z$ is a generalized L-space, 
\[v+h=(v_{g-2}^+)_*+(h_{g-2}^+)_*\co H_*(A_{g-2}^+)\to H_*(B^+)\]
is surjective. So 
\[HF^+(Z_0(K),[\widehat{F}],g-2)\cong\ker (v+h)
\]
as a $\mathbb Q[U]$--module.

Since $v$ is homogeneous and surjective,
there exists a homogeneous homomorphism $\rho\co H_*(B^+)\to H_*(A^+_{g-2})$ satisfying $$v\circ\rho=\mathrm{id}.$$
By (\ref{eq:vhDiff}) and the assumption that $g(F)>2$, the grading shift of $h$ is strictly less than the grading shift of $v$, so the grading shift of $\rho h$ is negative. As the grading of $H_*(A^+_{g-2})$ is bounded from below, for any $x\in H_*(A^+_{g-2})$, $(\rho h)^m(x)=0$ when $m$ is sufficiently large. So the map
$$\mathrm{id}-\rho h+(\rho h)^2-(\rho h)^3+\cdots\co H_*(A^+_{g-2})\to H_*(A^+_{g-2})$$
is well-defined, and it maps $\ker v$ to $\ker (v+h)$.

Assume that $a\in\ker v$ is a homogeneous element with $Ua\ne0$. Then
\[
a'=(\mathrm{id}-\rho h+(\rho h)^2-(\rho h)^3+\cdots)(a)=a+\text{lower grading terms}\in\ker(v+h)
\]
so 
\[
Ua'=Ua+\text{lower grading terms}
\]
which is nonzero since $Ua\ne0$.
\end{proof}

We will use the following elementary lemma in linear algebra.

\begin{lem}\label{lem:Tensor}
Let $\tensor V,\tensor W$ be two linear spaces over a field $\mathbb F$, and let $\tensor V_1,\tensor W_1$ be their subspaces, respectively. If $v\in \tensor V\setminus\tensor V_1$, $w\in \tensor W\setminus\tensor W_1$, then 
\[v\otimes w\notin \tensor V_1\otimes\tensor W+\tensor V\otimes\tensor W_1.\]
\end{lem}
\begin{proof}
Suppose that $\dim \tensor V=m$, $\dim \tensor V_1=m_1$, $\dim \tensor W=n$, $\dim \tensor W_1=n_1$. 
We can choose a basis 
\[
v_1,\dots,v_m
\]
of $\tensor V$, such that $v_1,\dots,v_{m_1}$ is a basis of $\tensor V_1$, and $v=v_{m_1+1}$. Similarly, we choose a basis 
\[
w_1,\dots,w_n
\]
of $\tensor W$, such that $w_1,\dots,w_{n_1}$ is a basis of $\tensor W_1$, and $w=w_{n_1+1}$. Then
\[
v_i\otimes w_j, 1\le i\le m, 1\le j\le n
\]
is a basis for $\tensor V\otimes\tensor W$.
Now $\tensor V_1\otimes\tensor W+\tensor V\otimes\tensor W_1$ is spanned by
\[
v_i\otimes w_j, 1\le i\le m_1 \text{ or } 1\le j\le n_1.
\]
So $v\otimes w=v_{m_1+1}\otimes w_{n_1+1}$ is not in this subspace.
\end{proof}

Let $\partial$ be the differential in $C=CFK^{\infty}$, $\partial_0$ be the component of $\partial$ which preserves the $(i,j)$--grading, $\partial_z$ be the component of $\partial$ which decreases the $(i,j)$--grading by $(0,1)$, and $\partial_w$ be the component which decreases the $(i,j)$--grading by $(1,0)$.
Since $\partial^2=0$, we have
\begin{equation}\label{eq:dChain}
\partial_{z}\circ\partial_0+\partial_0\circ\partial_{z}=0,\quad \partial_{w}\circ\partial_0+\partial_0\circ\partial_{w}=0,
\end{equation} 
and
\begin{equation}\label{eq:d^2}
\partial_w\circ\partial_z+\partial_{zw}\circ\partial_0+\partial_0\circ\partial_{zw}=0\quad \text{on } C(0,g).
\end{equation} 
It follows from (\ref{eq:dChain}) that $\partial_z$ and $\partial_w$ induces homomorphisms on the homology with respect to the differential $\partial_0$, denoted by $(\partial_z)_*$ and $(\partial_w)_*$.
By (\ref{eq:d^2}), 
\begin{equation}\label{eq:dwdz}
(\partial_w)_*\circ(\partial_z)_*=0
\end{equation}
 on $H_*( C(0,g))$.

\begin{thm}\label{thm:General}
Let $Z$ be a generalized L-space, $K\subset Z$ be a null-homologous knot. Let $F$ be a Seifert surface of $K$ with genus $g>2$. Let $d\in\mathbb Q$ satisfy 
\begin{equation}\label{eq:d+1Trivial}
\widehat{HFK}_{d\pm1}(Z,K,[F],g)=0.
\end{equation}
If there exist two elements $\gamma_1,\gamma_2\in H_1(F)$ with $\gamma_1\cdot\gamma_2\ne0$, such that the images of $\gamma_1,\gamma_2$ in $H_1(Z)$ are linearly dependent, then 
\[
\mathrm{rank}\widehat{HFK}_{d}(Z,K,[F],g)\le\mathrm{rank}\widehat{HFK}_{d-1}(Z,K,[F],g-1).
\]
\end{thm}
\begin{proof}
The chain complex
$C\{i<0, j\ge g-2\}$ has the form
\begin{equation}\label{eq:MCg-2}
\xymatrixcolsep{5pc}
\begin{xymatrix}{  
& C(-1,g-1)\ar[d]^{\partial_z}\ar[dl]_{\partial_{zw}}\\
C(-2,g-2)&C(-1,g-2)\ar[l]_{\partial_w},}
\end{xymatrix}
\end{equation}
where
\[C_{*-2}(-1,g-1)\cong C_{*-4}(-2,g-2)\cong\widehat{CFK}_*(Z,K,[F],g),\] 
and
\[ C_{*-2}(-1,g-2)\cong\widehat{CFK}_*(Z,K,[F],g-1).
\]

By abuse of notation, we will use $\partial_z$ and $\partial_w$ to denote their restrictions
\[\partial_z\co \widehat{CFK}_{d}(Z,K,[F],g)\to \widehat{CFK}_{d-1}(Z,K,[F],g-1)\] 
and \[\partial_w\co \widehat{CFK}_{d-1}(Z,K,[F],g-1)\to \widehat{CFK}_{d}(Z,K,[F],g).\]

Using (\ref{eq:dwdz}), we have
\begin{eqnarray*}
&&\mathrm{rank}\ker(\partial_z)_*\\
&=&\mathrm{rank}\widehat{HFK}_d(Z,K,[F],g)-\mathrm{rank}\:\im(\partial_z)_*\\
&\ge&\mathrm{rank}\widehat{HFK}_d(Z,K,[F],g)-\mathrm{rank}\ker(\partial_w)_*\\
&=&\mathrm{rank}\widehat{HFK}_d(Z,K,[F],g)-\mathrm{rank}\widehat{HFK}_{d-1}(Z,K,[F],g-1)+\mathrm{rank}\:\im(\partial_w)_*.
\end{eqnarray*}
If 
\begin{equation}\label{eq:RankIneq}
\mathrm{rank}\widehat{HFK}_d(Z,K,[F],g)>\mathrm{rank}\widehat{HFK}_{d-1}(Z,K,[F],g-1),
\end{equation}
 then
\[
\mathrm{rank}\ker(\partial_z)_*>\mathrm{rank}\:\im(\partial_w)_*,
\]
so there exists an element $x\in \ker(\partial_z)_*$, such that $Ux\notin\im(\partial_w)_*$. 
Let $\xi\in C_{d-2}(-1,g-1)$ be a closed chain representing $x$, then $\partial_z(\xi)$ is an exact chain in $C_{d-3}(-1,g-2)$. So there exists an element $\eta\in C_{d-2}(-1,g-2)$ with $\partial_0 \eta=\partial_z(\xi)$. By (\ref{eq:dChain}) and (\ref{eq:d^2}), 
\[
\partial_0\partial_w \eta=-\partial_w\partial_0 \eta=-\partial_w\partial_z(\xi)=\partial_0\partial_{zw}(\xi).
\]
So $\partial_w\eta-\partial_{zw}(\xi)$ is a closed chain in $C_{d-3}(-2,g-2)\cong\widehat{CFK}_{d+1}(Z,K,[F],g)$. By (\ref{eq:d+1Trivial}), $\partial_w\eta-\partial_{zw}(\xi)$ is exact, so there exists an element $\zeta\in C_{d-2}(-2,g-2)$ with $\partial_0\zeta=\partial_w\eta-\partial_{zw}(\xi)$. This means that $\xi-\eta+\zeta$ is a cycle in the mapping cone (\ref{eq:MCg-2}).

Now we want to prove $U(\xi-\eta+\zeta)=U\xi$ is not exact in (\ref{eq:MCg-2}). Otherwise, assume
\begin{equation}\label{eq:Uxi}
U\xi=\partial(\xi'+\eta'+\zeta'),
\end{equation} 
where 
\[\xi'\in C_{d-3}(-1,g-1), \eta'\in C_{d-3}(-1,g-2), \zeta'\in C_{d-3}(-2,g-2).\]
Considering the components of (\ref{eq:Uxi}), we get
\begin{eqnarray}
0&=&\partial_0\xi',\label{eq:xi'}\\
0&=& \partial_z\xi'+\partial_0\eta',\label{eq:eta'}\\ 
U\xi &=&\partial_{zw}\xi'+\partial_w\eta'+\partial_0\zeta'.\label{eq:zeta'}
\end{eqnarray}
By (\ref{eq:xi'}), $\xi'$ is a cycle in $C_{d-3}(-1,g-1)\cong\widehat{CFK}_{d-1}(Z,K,[F],g)$. By (\ref{eq:d+1Trivial}), $\xi'$ is exact, so there exists $\omega\in  C_{d-2}(-1,g-1)$ with $\partial_0\omega=\xi'$.
Using (\ref{eq:dChain}) and (\ref{eq:eta'}), we get
\[
\partial_0(\eta'-\partial_z\omega)=0.
\]
Using (\ref{eq:d^2}) and (\ref{eq:zeta'}), we get
\[
U\xi=-\partial_0\partial_{zw}\omega+\partial_w(\eta'-\partial_z\omega)+\partial_0\zeta',
\]
which means that $U\xi$ is homologous to an element in $\partial_w(\ker\partial_0)$.
Since $[U\xi]=Ux\notin\im(\partial_w)_*$, we get a contradiction.

Now we have proved that $U\ne0$ in the mapping cone (\ref{eq:MCg-2}).
By Proposition~\ref{prop:ZeroSurg}, we have
$U\ne0$ in $HF^+(Z_0(K),[\widehat F],g-2)$, a contradiction to Theorem~\ref{thm:ClosedSecond}.
\end{proof}

\begin{proof}[Proof of Theorem~\ref{thm:SecondTerm}]
When $g>2$, this follows from Theorem~\ref{thm:General}.

If $g=2$, we assume (\ref{eq:RankIneq}) holds.
As in the proof of Theorem~\ref{thm:General}, there exists an element $x\in \ker(\partial_z)_*$, such that $Ux\notin\im(\partial_w)_*$. Consider the element 
\[x\otimes x\in \widehat{HFK}_d(Z,K,[F],g)\otimes\widehat{HFK}_d(Z,K,g)\cong \widehat{HFK}_{2d}(Z\#Z,K\#K,[F\natural F],2g).\]
In the complex $CFK^{\infty}(Z\#Z,K\#K)$, we can check $x\otimes x\in  \ker(\partial_z)_*$, while $U(x\otimes x)\notin\im(\partial_w)_*$ by Lemma~\ref{lem:Tensor}. Let $\gamma_1,\gamma_2$ be a pair of elements in $H_1(F)$ with $\gamma_1\cdot\gamma_2\ne0$. We can think of $\gamma_1,\gamma_2$ as elements in the first summand of $H_1(F\natural F)\cong H_1(F)\oplus H_1(F)$. Then the images of $\gamma_1,\gamma_2$ in $H_1(Z\#Z)$ are linearly dependent. So we can apply Theorem~\ref{thm:ClosedSecond} to get a contradiction as in the proof of Theorem~\ref{thm:General}.

The case $g=1$ can be proved similarly by considering a three-fold connected sum.
\end{proof}


\begin{thebibliography}{H}

\bib{BVV}{article}{
   author={Baldwin, John},
   author={Vela-Vick, David Shea},
   title={A note on the knot Floer homology of fibered knots},
   journal={Algebr. Geom. Topol.},
   volume={18},
   date={2018},
   number={6},
   pages={3669--3690},
   issn={1472-2747},
}

\bib{Gh}{article}{
   author={Ghiggini, Paolo},
   title={Knot Floer homology detects genus-one fibred knots},
   journal={Amer. J. Math.},
   volume={130},
   date={2008},
   number={5},
   pages={1151--1169},
}

\bib{HeddenNi}{article}{
   author={Hedden, Matthew},
   author={Ni, Yi},
   title={Khovanov module and the detection of unlinks},
   journal={Geom. Topol.},
   volume={17},
   date={2013},
   number={5},
   pages={3027--3076},
   issn={1465-3060},
}


\bib{HW}{article}{
   author={Hedden, Matthew},
   author={Watson, Liam},
   title={On the geography and botany of knot Floer homology},
   journal={Selecta Math. (N.S.)},
   volume={24},
   date={2018},
   number={2},
   pages={997--1037},
}




\bib{MO}{article}{
   author={Manolescu, Ciprian},
   author={Ozsv\'{a}th, Peter},
   title={Heegaard Floer homology and integer surgeries on links},
   status={preprint},
   date={2010},
   eprint={https://arxiv.org/abs/1011.1317}
}


\bib{NiFibred}{article}{
   author={Ni, Yi},
   title={Knot Floer homology detects fibred knots},
   journal={Invent. Math.},
   volume={170},
   date={2007},
   number={3},
   pages={577--608},
}

\bib{NiNormCosmetic}{article}{
   author={Ni, Yi},
   title={Thurston norm and cosmetic surgeries},
   conference={
      title={Low-dimensional and symplectic topology},
   },
   book={
      series={Proc. Sympos. Pure Math.},
      volume={82},
      publisher={Amer. Math. Soc., Providence, RI},
   },
   date={2011},
   pages={53--63},
}

\bib{NiPropG}{article}{
   author={Ni, Yi},
   title={Property G and the $4$--genus},
   status={preprint},
   date={2020},
   eprint={https://arxiv.org/abs/2007.03721}
}

\bib{OSzAbGr}{article}{
   author={Ozsv\'{a}th, Peter},
   author={Szab\'{o}, Zolt\'{a}n},
   title={Absolutely graded Floer homologies and intersection forms for
   four-manifolds with boundary},
   journal={Adv. Math.},
   volume={173},
   date={2003},
   number={2},
   pages={179--261},
   issn={0001-8708},
}

\bib{OSzAlternating}{article}{
   author={Ozsv\'{a}th, Peter},
   author={Szab\'{o}, Zolt\'{a}n},
   title={Heegaard Floer homology and alternating knots},
   journal={Geom. Topol.},
   volume={7},
   date={2003},
   pages={225--254},
   issn={1465-3060},
}

\bib{OSzAnn1}{article}{
   author={Ozsv\'{a}th, Peter},
   author={Szab\'{o}, Zolt\'{a}n},
   title={Holomorphic disks and topological invariants for closed
   three-manifolds},
   journal={Ann. of Math. (2)},
   volume={159},
   date={2004},
   number={3},
   pages={1027--1158},
}


\bib{OSzAnn2}{article}{
   author={Ozsv\'{a}th, Peter},
   author={Szab\'{o}, Zolt\'{a}n},
   title={Holomorphic disks and three-manifold invariants: properties and
   applications},
   journal={Ann. of Math. (2)},
   volume={159},
   date={2004},
   number={3},
   pages={1159--1245},
   issn={0003-486X},
}


\bib{OSzSympl}{article}{
   author={Ozsv\'{a}th, Peter},
   author={Szab\'{o}, Zolt\'{a}n},
   title={Holomorphic triangle invariants and the topology of symplectic
   four-manifolds},
   journal={Duke Math. J.},
   volume={121},
   date={2004},
   number={1},
   pages={1--34},
   issn={0012-7094},
}

\bib{OSzKnot}{article}{
   author={Ozsv\'{a}th, Peter},
   author={Szab\'{o}, Zolt\'{a}n},
   title={Holomorphic disks and knot invariants},
   journal={Adv. Math.},
   volume={186},
   date={2004},
   number={1},
   pages={58--116},
}

\bib{OSzGenus}{article}{
   author={Ozsv\'{a}th, Peter},
   author={Szab\'{o}, Zolt\'{a}n},
   title={Holomorphic disks and genus bounds},
   journal={Geom. Topol.},
   volume={8},
   date={2004},
   pages={311--334},
}


\bib{OSzFour}{article}{
   author={Ozsv\'{a}th, Peter},
   author={Szab\'{o}, Zolt\'{a}n},
   title={Holomorphic triangles and invariants for smooth four-manifolds},
   journal={Adv. Math.},
   volume={202},
   date={2006},
   number={2},
   pages={326--400},
   issn={0001-8708},
}

\bib{OSzIntSurg}{article}{
   author={Ozsv\'{a}th, Peter},
   author={Szab\'{o}, Zolt\'{a}n},
   title={Knot Floer homology and integer surgeries},
   journal={Algebr. Geom. Topol.},
   volume={8},
   date={2008},
   number={1},
   pages={101--153},
   issn={1472-2747},
}

\bib{RasThesis}{book}{
   author={Rasmussen, Jacob Andrew},
   title={Floer homology and knot complements},
   note={Thesis (Ph.D.)--Harvard University},
   publisher={ProQuest LLC, Ann Arbor, MI},
   date={2003},
   pages={126},
}



\bib{Wu}{article}{
   author={Wu, Zhongtao},
   title={$U$-action on perturbed Heegaard Floer homology},
   journal={J. Symplectic Geom.},
   volume={10},
   date={2012},
   number={3},
   pages={423--445},
   issn={1527-5256},
}



\end{thebibliography}
\end{document}